\documentclass[12pt]{article}

\usepackage{amsfonts,amssymb,amsmath,amsthm}
\usepackage{graphics,graphicx}                 
\usepackage{epsfig}

\usepackage{tikz}
\usetikzlibrary{matrix,arrows}
\usetikzlibrary{positioning}
\usetikzlibrary{fit}
\usetikzlibrary{patterns}
\usetikzlibrary{decorations.markings}

\newtheorem{theorem}{Theorem}
\newtheorem{lemma}[theorem]{Lemma}

\newtheorem{example}[theorem]{Example}
\newtheorem{proposition}[theorem]{Proposition}
\newtheorem{corollary}[theorem]{Corollary}

\long\def\symbolfootnote[#1]#2{\begingroup
\def\thefootnote{\fnsymbol{footnote}}\footnote[#1]{#2}\endgroup}

\DeclareMathOperator{\Av}{\mathrm{Av}}
\DeclareMathOperator{\hypergeom}{\mathrm{hypergeom}}

\newcommand{\fl}{\mathrm{fl}}

\newcommand{\pattern}[4]{										
  \raisebox{0.6ex}{
  \begin{tikzpicture}[scale=0.35, baseline=(current bounding box.center), #1]
  \useasboundingbox (0.0,-0.1) rectangle (#2+1.4,#2+1.1);
    \foreach \x/\y in {#4}
      \fill[pattern=north east lines] (\x,\y) rectangle +(1,1);
    \draw (0.01,0.01) grid (#2+0.99,#2+0.99);
    \foreach \x/\y in {#3}
      \filldraw (\x,\y) circle (6pt);
  \end{tikzpicture}}
}

\title{Restricted non-separable planar maps and some pattern avoiding permutations}

\author{Sergey Kitaev\thanks{Supported by a grant from Iceland, Liechtenstein and Norway through the EEA Financial Mechanism. Supported and coordinated by Universidad Complutense de Madrid.}
\and
Pavel Salimov\thanks{Supported by the Russian Foundation for Basic Research grant no.\ 10-01-00616 and by the Icelandic Research Fund grant no.\ 090038013.}
\and
Christopher Severs\thanks{Supported by grant no.\ 090038013 from the Icelandic Research Fund.}
\and
Henning Ulfarsson\thanks{Supported by grant no.\ 090038013-4 from the Icelandic Research Fund.}}

\begin{document}

\maketitle

\begin{abstract}
Tutte founded the theory of enumeration of planar maps in a series of papers in the
1960s. Rooted non-separable planar maps are in bijection with West-$2$-stack-sortable
permutations, $\beta(1,0)$-trees introduced by Cori, Jacquard and Schaeffer in 1997,
as well as a family of permutations defined by the avoidance of two four letter
patterns. In this paper we give upper and lower bounds on the number of
multiple-edge-free rooted non-separable planar maps. We also use the
bijection between rooted non-separable planar maps and a certain class of permutations, found by Claesson, Kitaev and Steingrimsson in 2009, to show that
the number of $2$-faces (excluding the root-face) in a map equals the number of occurrences of a certain mesh pattern in the permutations. We further show
that this number is also the number of nodes in the corresponding $\beta(1,0)$-tree
that are single children with maximum label. Finally, we give asymptotics for some
of our enumerative results.
\end{abstract}

\section{Introduction}

A \emph{map} is a partition of a compact oriented surface into three
finite sets: a set of \emph{vertices}, a set of \emph{edges}
and a set of \emph{faces} (disjoint simply connected domains). Maps
are considered up to orientation-preserving homeomorphisms of the
surface. In this paper we deal with classical \emph{planar maps}
considered, for example, by Tutte~\cite{T1} who founded the
theory of their enumeration in a series of papers in the 1960s
(see~\cite{CS} for references). The maps considered by us are {\em
rooted}, meaning that a directed edge is distinguished as the root, which defines a \emph{root-vertex}
and a \emph{root-face} (the face located to the right while traversing the root following its direction). For convenience we assume that the root-face of a map coincides with the map's outer face.

A \emph{cut vertex} in a map $M$ is a vertex $v$ such that there exists a partition of
the edge set of $M$ into two subsets such that $v$ is the unique
vertex which is incident with edges of the two subsets. A planar map
is \emph{non-separable} if it has no loops and no cut vertices.
From here on we will use \emph{maps} to mean \emph{rooted non-separable planar maps} for the sake of brevity. All maps on four edges are given in Figure~\ref{maps}.

\begin{figure}[ht]
\begin{center}
\begin{tikzpicture}[scale = 0.85, place/.style = {circle,draw = black!50,fill = gray!20,thick,inner sep=3pt}, auto,
  decoration={
    markings,
    mark=at position 0.5 with {\arrow[black]{stealth};}}
    ]

 \node [place] (L) at (180:1)      {};
 \node [place] (R) at (0:1)      {};
 
 \draw [-,semithick, bend left=40] (L) to (R);
 \draw [-,semithick, postaction={decorate}, bend right=80] (R) to (L);
 \draw [-,semithick, bend right=40] (L) to (R);
 \draw [-,semithick, bend right=80] (L) to (R);
 
 \begin{scope}[xshift = 3cm]
 \node [place] (L) at (180:1)     {};
 \node [place] (M) at (0,0)     {};
 \node [place] (R) at (0:1)    {};
 
 \draw [-,semithick, bend left=40]              (L) to (R);
 \draw [-,semithick]                      (L) to (M);
 \draw [-,semithick]                      (M) to (R);
 \draw [-,semithick, postaction={decorate}, bend right=40]  (L) to (R);
 \end{scope}
 
 \begin{scope}[xshift = 6cm]
 \node [place] (L) at (180:1)     {};
 \node [place] (M) at (0,0)     {};
 \node [place] (R) at (0:1)    {};
 
 \draw [-,semithick]                      (L) to (M);
 \draw [-,semithick]                      (M) to (R);
 \draw [-,semithick, bend right=40]             (R) to (L);
 \draw [-,semithick, bend right=40,, postaction={decorate}]   (M) to (R);
 \end{scope}

 \begin{scope}[yshift = -2cm]
 \node [place] (L) at (180:1)     {};
 \node [place] (M) at (0,0)     {};
 \node [place] (R) at (0:1)    {};
 
 \draw [-,semithick]                      (L) to (M);
 \draw [-,semithick]                      (M) to (R);
 \draw [-,semithick, bend right=40]             (M) to (R);
 \draw [-,semithick, bend right=40, postaction={decorate}]    (R) to (L);
 \end{scope}
 
 \begin{scope}[xshift = 3cm, yshift = -2cm]
 \node [place] (L) at (180:1)     {};
 \node [place] (M) at (0,0)     {};
 \node [place] (R) at (0:1)    {};
 
 \draw [-,semithick, bend right = 40]             (L) to (M);
 \draw [-,semithick]                      (L) to (M);
 \draw [-,semithick]                      (M) to (R);
 \draw [-,semithick, bend right=40, postaction={decorate}]    (R) to (L);
 \end{scope}
 
 \begin{scope}[xshift = 6cm, yshift = -2cm]
 \node [place] (L) at (-1,0)     {};
 \node [place] (M1) at (0,0)     {};
 \node [place] (M2) at (1,0)     {};
 \node [place] (R) at (2,0)    {};
 
 \draw [-,semithick]                      (L) to (M1);
 \draw [-,semithick]                      (M1) to (M2);
 \draw [-,semithick]                      (M2) to (R);
 \draw [-,semithick, bend right=40, postaction={decorate}]    (R) to (L);
 \end{scope}
 
\end{tikzpicture}
 \caption{All rooted non-separable planar maps on four edges}
 \label{maps}
\end{center}
\end{figure}
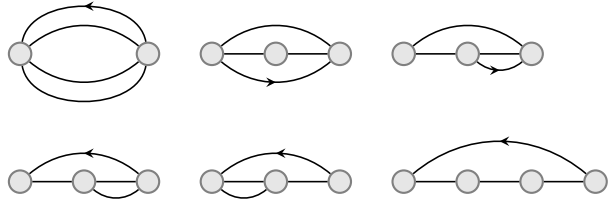

West~\cite{W90} defined and described permutations sortable by two passes through
a stack (now often called \emph{West-$2$-stack-sortable} permutations). He also
conjectured that the number of such permutations of length $n$ is the same as the
number of maps on $n+1$ edges. This conjecture was
first proven by Zeilberger~\cite{Z}. Dulucq et al.~\cite{DGW,DGG} and Goulden and
West~\cite{GW} later found combinatorial proofs. These maps also have connections to
theoretical physics, where they are used as a discrete model for 2D quantum gravity, see e.g., Schaeffer and Zinn-Justin~\cite{SZ}. Finally they have connections to pattern-avoiding permutations (see, for example, Kitaev~\cite[Chapter 2.2]{S}).

The number of rooted non-separable planar maps on $n+1$ edges was
first determined by Tutte~\cite{T} and is given by
\[
\frac{4(3n)!}{n!(2n+2)!}.
\]
This enumeration was also proved differently by
Brown~\cite{B}. Moreover, Zeilberger~\cite{Z} showed that the generating function $A(x)$ for the maps is given by $A(x) = 2 + xB(x)$, where $B(x)$ satisfies the functional
equation
\[
B(x) = 1 - 8x + 2x(5-6x)B(x) - 2x^2(1+3x)B^2(x) - x^4B^3(x).
\]
Alternatively, we can express this in terms of a \emph{hypergeometric function}. Let $\mathbf{a} = (a_1, \dotsc, a_p)$, $\mathbf{b} = (b_1, \dotsc, b_q)$ be vectors of real numbers. Then $\hypergeom(\mathbf{a},\mathbf{b},z)$ is the function whose power series is
\[
\sum_{k \geq 0} \frac{a_1^{\bar k} \dotsm a_p^{\bar k}}{b_1^{\bar k} \dotsm b_q^{\bar k}} \frac{z^k}{k!},
\]
and $a^{\bar k}$ is the rising factorial $a(a+1)\dotsm(a+k-1)$. Then the generating function for rooted non-separable planar maps is given by
\[
A(x)=\frac{2}{3x}\left(\hypergeom\left(\left[-\frac{2}{3}, -\frac{1}{3}\right],\left[\frac{1}{2}\right],\frac{27x}{4}\right)-1\right),
\]
see~\cite[A000139]{OEIS}. From the functional equation one can derive the asymptotic
estimate
\[
[x^n]A(x) \sim \frac{2}{27}\sqrt{\frac{3}{\pi n^5}} \left(\frac{27}{4}\right)^n,
\]
see Flajolet and Sedgewick~\cite{FS}, Example IX.42.

Cori et al.~\cite{CJS} introduced \emph{description trees}, thus placing under one roof several classes of planar maps. In particular, $\beta(1,0)$-trees (defined below) are in one-to-one correspondence with rooted non-separable planar maps. Description trees may be useful in dealing with the corresponding planar maps. For example, Claesson et al.~\cite{CKD} obtained non-trivial equidistribution results on so-called \emph{bicubic maps} using $\beta(0,1)$-trees.
In this paper, our main focus is understanding the structure of $\beta(1,0)$-trees corresponding to certain restricted maps and deriving some corollaries of that structure. In Section~\ref{sec:primmaps} we consider maps without internal $2$-faces,
enumerate them in Theorem~\ref{enumPrimMaps} and find the corresponding property of
$\beta(1,0)$-trees in Proposition~\ref{aProp1}. In Section~\ref{sec:mapsrestr} we use the enumeration of
internal $2$-face-free maps to derive the enumeration of $2$-face-free maps. This
provides an upper bound for multiple-edge-free maps which is asymptotically $5.75^n$.
We then use $\beta(1,0)$-trees with labels restricted to $\{1,2,3\}$ to find a lower bound which is asymptotically $4.24^n$.
In Section~\ref{sec:patts} we show that internal $2$-face-free maps correspond to
permutations avoiding two four letter patterns and one mesh pattern on two letters.
In Theorem~\ref{thm:perms-dist} we show that occurrences of the mesh pattern equal the
number of $2$-faces (excluding the root-face) in a map, and the number of nodes
in a $\beta(1,0)$-tree that are a single child with maximum label.

We believe that our studies will be of help in future research on rooted non-separable planar maps, for example, from an enumerative point of view.

\section{Preliminaries on $\beta(1,0)$-trees and the bijection to maps}
A \emph{$\beta(1,0)$-tree} is a rooted planar tree labeled with positive integers such that the leaves have label $1$,
the root has label equal to the sum of its children's labels and any other node has label no greater than the sum of its children's labels.
All $\beta(1,0)$-trees on three edges are presented in Figure~\ref{beta10}.
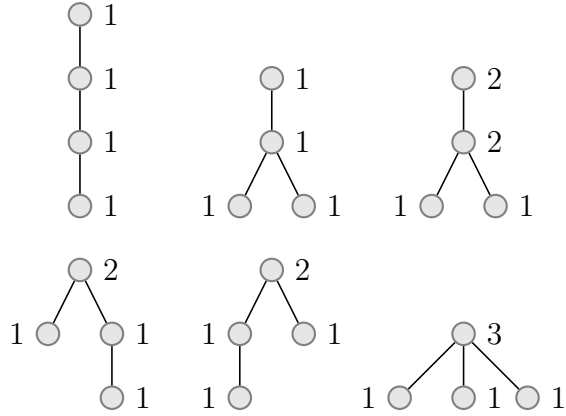
\begin{figure}[ht]
\begin{center}
\begin{tikzpicture}[scale = 0.85, place/.style = {circle,draw = black!50,fill = gray!20,thick,inner sep=3pt}, auto]

 \node [place] (1) at (0,2)  [label = right:$1$]     {};
 \node [place] (2) at (0,1)  [label = right:$1$]       {};
 \node [place] (3) at (0,0)  [label = right:$1$]       {};
 \node [place] (4) at (0,-1) [label = right:$1$]       {};
 
 \draw [-,semithick]      (1) to (2);
 \draw [-,semithick]      (2) to (3);
 \draw [-,semithick]      (3) to (4);
 
 \begin{scope}[xshift = 3cm]
 \node [place] (1) at (0,1)  [label = right:$1$]     {};
 \node [place] (2) at (0,0)  [label = right:$1$]       {};
 \node [place] (3) at (0.5,-1)  [label = right:$1$]      {};
 \node [place] (4) at (-0.5,-1) [label = left:$1$]       {};
 
 \draw [-,semithick]      (1) to (2);
 \draw [-,semithick]      (2) to (3);
 \draw [-,semithick]      (2) to (4);
 \end{scope}
 
 \begin{scope}[xshift = 6cm]
 \node [place] (1) at (0,1)  [label = right:$2$]     {};
 \node [place] (2) at (0,0)  [label = right:$2$]       {};
 \node [place] (3) at (0.5,-1)  [label = right:$1$]      {};
 \node [place] (4) at (-0.5,-1) [label = left:$1$]       {};
 
 \draw [-,semithick]      (1) to (2);
 \draw [-,semithick]      (2) to (3);
 \draw [-,semithick]      (2) to (4);
 \end{scope}

 \begin{scope}[xshift = 0cm, yshift = -3cm]
 \node [place] (1) at (0,1)  [label = right:$2$]     {};
 \node [place] (2) at (-0.5,0)  [label = left:$1$]       {};
 \node [place] (3) at (0.5,0)  [label = right:$1$]       {};
 \node [place] (4) at (0.5,-1) [label = right:$1$]       {};
 
 \draw [-,semithick]      (1) to (2);
 \draw [-,semithick]      (1) to (3);
 \draw [-,semithick]      (3) to (4);
 \end{scope}
 
 \begin{scope}[xshift = 3cm, yshift = -3cm]
 \node [place] (1) at (0,1)  [label = right:$2$]     {};
 \node [place] (2) at (-0.5,0)  [label = left:$1$]       {};
 \node [place] (3) at (0.5,0)  [label = right:$1$]       {};
 \node [place] (4) at (-0.5,-1) [label = left:$1$]       {};
 
 \draw [-,semithick]      (1) to (2);
 \draw [-,semithick]      (1) to (3);
 \draw [-,semithick]      (2) to (4);
 \end{scope}
 
 \begin{scope}[xshift = 6cm, yshift = -3cm]
 \node [place] (1) at (0,0)  [label = right:$3$]     {};
 \node [place] (2) at (-1,-1)  [label = left:$1$]      {};
 \node [place] (3) at (0,-1)  [label = right:$1$]      {};
 \node [place] (4) at (1,-1) [label = right:$1$]       {};
 
 \draw [-,semithick]      (1) to (2);
 \draw [-,semithick]      (1) to (3);
 \draw [-,semithick]      (1) to (4);
 \end{scope}
 
\end{tikzpicture}
 \caption{All $\beta(1,0)$-trees on three edges}
 \label{beta10}
\end{center}
\end{figure}

We refer to~\cite{CJS} for more details on the bijection between $\beta(1,0)$-trees and the maps in question. Here we provide a short description illustrated by the example in Figure~\ref{bijection}. The basic strategy is to define maps corresponding to leaves, and to find the map corresponding to an internal node once all the maps corresponding to its children are determined; at the last step, we find the map corresponding to the root, once the maps corresponding to the root's children are found. More precisely, begin with assigning to the leaves of a given tree the maps as in Figure~\ref{bijection}, where $R$ indicates the root vertex and the star is an auxiliary vertex mark that will disappear at the end of the procedure.

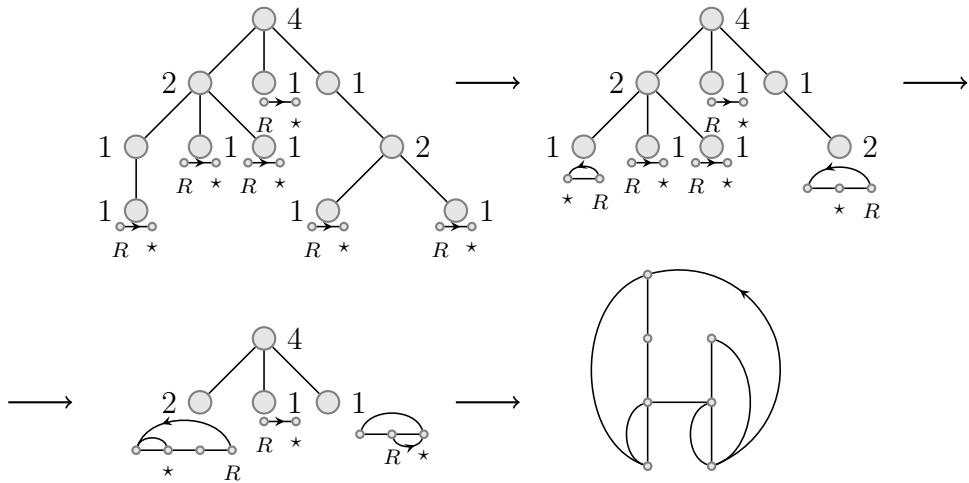
\begin{figure}[ht]
\begin{center}
\begin{tikzpicture}[scale = 0.85,
    place/.style = {circle,draw = black!50,fill = gray!20,thick,inner sep=3pt},
    mini/.style = {circle,draw = black!50,fill = gray!20,thick,inner sep=1pt},
    auto,
    decoration={
    markings,
    mark=at position 0.7 with {\arrow[black]{stealth};}}]
    
 \draw [->, thick] (3,2) to (4,2);

 \node [place] (r) at (0,3)  [label = right:$4$]     {};
 \node [place] (1lv1) at (-1,2)  [label = left:$2$]      {};
 \node [place] (1lv2) at (0,2)  [label = right:$1$]      {};
 \node [place] (1lv3) at (1,2) [label = right:$1$]       {};
 
 \node [place] (2lv1) at (-2,1)  [label = left:$1$]      {};
 \node [place] (2lv2) at (-1,1)  [label = right:$1$]       {};
 \node [place] (2lv3) at (0,1) [label = right:$1$]       {};
 
 \node [place] (2lv4) at (2,1) [label = right:$2$]       {};
 
 \node [place] (3lv1) at (-2,0) [label = left:$1$]       {};
 \node [place] (3lv2) at (1,0) [label = left:$1$]      {};
 \node [place] (3lv3) at (3,0) [label = right:$1$]       {};

 \draw [-,semithick]      (r) to (1lv1);
 \draw [-,semithick]      (r) to (1lv2);
 \draw [-,semithick]      (r) to (1lv3);
 
 \draw [-,semithick]      (1lv1) to (2lv1);
 \draw [-,semithick]      (1lv1) to (2lv2);
 \draw [-,semithick]      (1lv1) to (2lv3);
 
 \draw [-,semithick]      (1lv3) to (2lv4);
 
 \draw [-,semithick]      (2lv1) to (3lv1);
 
 \draw [-,semithick]      (2lv4) to (3lv2);
 \draw [-,semithick]      (2lv4) to (3lv3);

 \foreach \x/\y in {0/1.7,-1.25/0.75,-0.25/0.75,-2.25/-0.25,0.75/-0.25,2.75/-0.25}
 {
  \begin{scope}[xshift = \x cm, yshift = \y cm]
  \node [mini] (a) at (0,0)  [label = below:$\scriptstyle{R}$]     {};
  \node [mini] (b) at (0.5,0)  [label = below:$\scriptstyle{\star}$]       {};
 
  \draw [-,semithick, postaction={decorate}]      (a) to (b);
  \end{scope}
 }

 \begin{scope}[xshift = 7cm]
 
 \draw [->, thick] (3,2) to (4,2);
 
 \node [place] (r) at (0,3)  [label = right:$4$]     {};
 \node [place] (1lv1) at (-1,2)  [label = left:$2$]      {};
 \node [place] (1lv2) at (0,2)  [label = right:$1$]      {};
 \node [place] (1lv3) at (1,2) [label = right:$1$]       {};
 
 \node [place] (2lv1) at (-2,1)  [label = left:$1$]      {};
 \node [place] (2lv2) at (-1,1)  [label = right:$1$]       {};
 \node [place] (2lv3) at (0,1) [label = right:$1$]       {};
 
 \node [place] (2lv4) at (2,1) [label = right:$2$]       {};

 \draw [-,semithick]      (r) to (1lv1);
 \draw [-,semithick]      (r) to (1lv2);
 \draw [-,semithick]      (r) to (1lv3);
 
 \draw [-,semithick]      (1lv1) to (2lv1);
 \draw [-,semithick]      (1lv1) to (2lv2);
 \draw [-,semithick]      (1lv1) to (2lv3);
 
 \draw [-,semithick]      (1lv3) to (2lv4);
 
 \begin{scope}[xshift = -2.25 cm, yshift = 0.5 cm]
 \node [mini] (a) at (0,0)  [label = below:$\scriptstyle{\star}$]     {};
 \node [mini] (b) at (0.5,0)  [label = below:$\scriptstyle{R}$]      {};
 
 \draw [-,semithick]      (a) to (b);
 \draw [-,semithick, postaction={decorate}, bend right = 70]      (b) to (a);
 \end{scope}
 
 \begin{scope}[xshift = 1.5 cm, yshift = 0.35 cm]
 \node [mini] (a) at (0,0)      {};
 \node [mini] (b) at (0.5,0)  [label = below:$\scriptstyle{\star}$]      {};
 \node [mini] (c) at (1,0)  [label = below:$\scriptstyle{R}$]      {};
 
 \draw [-,semithick]      (a) to (b);
 \draw [-,semithick]      (b) to (c);
 \draw [-,semithick, postaction={decorate}, bend right = 70]      (c) to (a);
 \end{scope}
 
 \foreach \x/\y in {0/1.7,-1.25/0.75,-0.25/0.75}
 {
  \begin{scope}[xshift = \x cm, yshift = \y cm]
  \node [mini] (a) at (0,0)  [label = below:$\scriptstyle{R}$]     {};
  \node [mini] (b) at (0.5,0)  [label = below:$\scriptstyle{\star}$]       {};
 
  \draw [-,semithick, postaction={decorate}]      (a) to (b);
  \end{scope}
 }

 \end{scope}
 
 \begin{scope}[yshift = -5cm]
 
 \draw [->, thick] (-4,2) to (-3,2);
 
 \draw [->, thick] (3,2) to (4,2);
 
 \node [place] (r) at (0,3)  [label = right:$4$]     {};
 \node [place] (1lv1) at (-1,2)  [label = left:$2$]      {};
 \node [place] (1lv2) at (0,2)  [label = right:$1$]      {};
 \node [place] (1lv3) at (1,2) [label = right:$1$]       {};
 
 \draw [-,semithick]      (r) to (1lv1);
 \draw [-,semithick]      (r) to (1lv2);
 \draw [-,semithick]      (r) to (1lv3);
 
 \begin{scope}[xshift = -2 cm, yshift = 1.25 cm]
 \node [mini] (a) at (0,0)       {};
 \node [mini] (b) at (0.5,0)  [label = below:$\scriptstyle{\star}$]      {};
 \node [mini] (c) at (1,0)         {};
 \node [mini] (d) at (1.5,0)  [label = below:$\scriptstyle{R}$]      {};

 \draw [-,semithick]      (a) to (b);
 \draw [-,semithick]      (b) to (c);
 \draw [-,semithick]      (c) to (d);
 \draw [-,semithick, bend left = 70]      (a) to (b);
 \draw [-,semithick, postaction={decorate}, bend right = 70]      (d) to (a);
 \end{scope}
 
 \begin{scope}[xshift = 1.5 cm, yshift = 1.5 cm]
 \node [mini] (a) at (0,0)      {};
 \node [mini] (b) at (0.5,0)  [label = below:$\scriptstyle{R}$]      {};
 \node [mini] (c) at (1,0)  [label = below:$\scriptstyle{\star}$]      {};
 
 \draw [-,semithick]      (a) to (b);
 \draw [-,semithick]      (b) to (c);
 \draw [-,semithick, bend right = 70]       (c) to (a);
 \draw [-,semithick, postaction={decorate}, bend right = 70]      (b) to (c);
 \end{scope}
 
  \foreach \x/\y in {0/1.7}
 {
  \begin{scope}[xshift = \x cm, yshift = \y cm]
  \node [mini] (a) at (0,0)  [label = below:$\scriptstyle{R}$]     {};
  \node [mini] (b) at (0.5,0)  [label = below:$\scriptstyle{\star}$]       {};
 
  \draw [-,semithick, postaction={decorate}]      (a) to (b);
  \end{scope}
 }

 \end{scope}
 
 \begin{scope}[xshift = 6cm, yshift = -4cm]

 \node [mini] (a) at (0,3)       {};
 \node [mini] (b) at (0,2)       {};
 \node [mini] (c) at (0,1)       {};
 \node [mini] (d) at (0,0)       {};
 
 \node [mini] (e) at (1,2)       {};
 \node [mini] (f) at (1,1)       {};
 \node [mini] (g) at (1,0)       {};

 \draw [-,semithick]      (a) to (b);
 \draw [-,semithick]      (b) to (c);
 \draw [-,semithick]      (c) to (d);
 \draw [-,semithick, bend right = 70]       (c) to (d);
 \draw [-,semithick, bend left = 70]      (d) to (a);
 
 \draw [-,semithick]      (e) to (f);
 \draw [-,semithick]      (f) to (g);
 \draw [-,semithick]      (f) to (c);
 \draw [-,semithick, bend right = 70]       (f) to (g);
 \draw [-,semithick, bend right = 70]       (g) to (e);
 
 \draw [-,semithick, bend right = 42.5, postaction={decorate}]      (g) to (2,2) to (a);
 
 \end{scope}
 
\end{tikzpicture}
 \caption{An example of mapping a $\beta(1,0)$-tree to a rooted non-separable planar map}
 \label{bijection}
\end{center}
\end{figure}

Assuming all the children of an internal node labelled $i$ are assigned maps numbered $1$ to $m$ from left-to-right (with two vertices on the outer face of each map labelled with $R$ and a star), take the star vertex of map $j$ and glue it to the root vertex of map $j+1$, for $j=1,2,\ldots,m-1$, and make the star vertex of map $m$ the new root pointing at the $R$ vertex of map 1 (this requires adding a new edge, the new root edge). Finally, in the obtained map, mark by star the $i$-th vertex counting from (but not including) the root vertex in the counterclockwise direction on the outer face. The only difference in performing the operation for the root of the given tree is that there is no need to put a star (which would otherwise be placed next to the root vertex in clockwise direction). It is not difficult to verify the correspondence of the statistics listed in Table~\ref{table:stats}.

\begin{table}[htdp]
\caption{Statistics preserved by the bijection}
\begin{center}
\begin{tabular}{l|l}
maps & trees \\
\hline
\# edges & \# nodes \\
\# vertices & \# leaves $+1$ \\
\# faces & \# internal nodes $+1$ \\
root-face degree & root-label $+1$ \\
\end{tabular}
\end{center}
\label{table:stats}
\end{table}%

The reader can check his/her understanding of the bijection by looking at Figures~\ref{maps} and \ref{beta10}: the trees considered there correspond to the maps in the same order under the bijection.

\section{Primitive maps and primitive $\beta(1,0)$-trees} \label{sec:primmaps}

In this section we define \emph{primitive maps} as internal $2$-face-free maps. Note that the root-face is allowed to be a $2$-face. The second and the last maps in Figure~\ref{maps} are primitive and the map constructed in Figure~\ref{bijection} is non-primitive.
Any map can be constructed from a primitive map by adding the appropriate multiplicities of edges, thus creating $2$-faces.

A $\beta(1,0)$-tree is said to be \emph{primitive} if it corresponds to a primitive map under the bijection above.
The word ``primitive'' is chosen in analogy with \emph{primitive $(2+2)$-free posets} studied by Dukes et al.~\cite{DKRS}. We can mimic the steps in that paper and use the generating function for all rooted non-separable planar maps $A(x)$ provided above to enumerate primitive maps. Namely, let $P(x)=\sum_{n\geq 0}p_nx^n$ be the generating function in question, where $p_n$ is the number of primitive maps on $n$ edges. Then the following result holds.
    \begin{theorem}\label{enumPrimMaps}
    We have
    \[
    P(x)=A\left(\frac{x}{x+1}\right)=\frac{2(x+1)}{3x}\left(\hypergeom\left(\left[-\frac{2}{3}, -\frac{1}{3}\right],\left[\frac{1}{2}\right],\frac{27x}{4(x+1)}\right)-1\right),
    \]
    with the asymptotic estimate
    \[
    [x^n]P(x) \sim \frac{46}{729}\sqrt{\frac{23}{\pi n^5}} \left(\frac{23}{4}\right)^n.
    \]
    \end{theorem}

\begin{proof} A primitive map having at least one edge gives rise to an infinite number of rooted non-separable planar maps by choosing multiplicities of edges, which can be recorded in terms of generating functions as follows:
\[
A(x)=\sum_{n\geq 0}p_n(x+x^2+\cdots)^n=\sum_{n\geq 0}p_n\left(\frac{x}{1-x}\right)^n=P\left(\frac{x}{1-x}\right).
\]
By substituting $x/(1-x)$ with $x$ we obtain the desired result. The dominant
singularity of $A$ is $\frac{4}{27}$, so the dominant singularity of the composition
is the solution of $\frac{x}{x+1} = \frac{4}{27}$, which
is $x = \frac{4}{23}$. Applying Theorem VI.4 and the expansion of $A$ given in
Example IX.42 in~\cite{FS} gives the claimed estimate.
\end{proof}

\begin{proposition}\label{aProp1}
A $\beta(1,0)$-tree is primitive if and only if it has
no node which has a single child with maximum label.
In fact, the number of internal $2$-faces in the corresponding map equals
the number of such nodes (not counting the root).
\end{proposition}

\begin{proof}
Consider a node $v$ in a $\beta(1,0)$-tree which has a single child $u$ with maximum label. Because $u$ has maximum label, the map corresponding to the subtree rooted at $u$ has its star vertex and root vertex separated by a single edge. Now, when we construct the map corresponding to the subtree rooted at $v$ we simply add a new edge from the star vertex to the root vertex of the previous map. These two vertices were already connected so the new edge completes an internal $2$-face. Now consider the step in the bijection between a tree and a map when an internal $2$-face is formed. Clearly this must be where a node $v$ has a single child $u$, since otherwise we would be adding an edge between two sub-maps having no common edges. Furthermore, for a $2$-face to form we must have had the star vertex and the root vertex in the sub-map at $u$ adjacent, but this requires the node $u$ to have maximum label.
\end{proof}


\section{Maps restricted in certain ways} \label{sec:mapsrestr}

In this section we explore $k$-face-free maps and multiple-edge-free maps.

\subsection{A generating function for $2$-face-free maps}\label{kfafm}
It is easily seen that a primitive map is $2$-face free if and only if the corresponding
$\beta(1,0)$-tree does not have root label $1$. Equivalently, the tree is \emph{decomposable}, that is the root has at least two children. Therefore the generating function for $2$-face-free maps is
\begin{align*}
P'(x) &= P(x)-xP(x) \\
&= \frac{2(1-x^2)}{3x}\left(\hypergeom\left(\left[-\frac{2}{3}, -\frac{1}{3}\right],\left[\frac{1}{2}\right],\frac{27x}{4(x+1)}\right)-1\right).
\end{align*}
The forbidden subtrees that define $2$-face-free maps are shown in Figure~\ref{fig:2-ff}.

\begin{figure}[ht]
\begin{center}
\begin{tikzpicture}[scale = 0.85,
    place/.style = {circle,draw = black!50,fill = gray!20,thick,inner sep=3pt},
    auto]

 \node           at (0,1.75) {$\vdots$};
 \node [place] (r1) at (0,1) {};
 \node [place] (n1) at (0,0) [label = right:${\scriptstyle\mathrm{max}}$] {};
 
 \node [place] (r2) at (3,1) [label = right:${\scriptstyle 1}$] {};
 \node [place] (n2) at (3,0) [label = right:${\scriptstyle 1}$] {};
 
 \draw [-,semithick]      (r1) to (n1);
 \draw [-,semithick]      (r2) to (n2);
 
 \draw [-,semithick]      (n1) to (0.75,-1) to (-0.75,-1) to (n1);
 \draw [-,semithick]      (n2) to (3.75,-1) to (2.25,-1) to (n2);
 
\end{tikzpicture}
 \caption{Forbidden subtrees defining $\beta(1,0)$-trees corresponding to $2$-face-free maps}
 \label{fig:2-ff}
\end{center}
\end{figure}
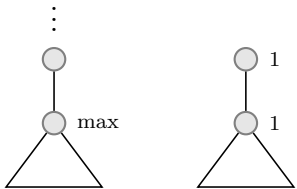

There is a natural extension, shown in Figure~\ref{fig:3-ff}, of these forbidden subtrees that can be used to characterize $\beta(1,0)$-trees corresponding to $3$-face-free
maps. A justification of the extension is straightforward and it is based on the bijection between $\beta(1,0)$-trees  and planar maps (similarly to the proof of Proposition~\ref{aProp1}). Indeed, thinking of the recursive way to build the bijection, one needs to make sure that adding a new edge on planar maps does not create an internal $3$-face, and also, at the very end, the outer face cannot be a triangle.  The two left-most subtrees in Figure~\ref{fig:3-ff} guarantee that no internal $3$-face emerges (while passing from the nodes containing the ``max" symbol to their parents), while the two right-most (sub)trees guarantee that the outer-face is not a triangle. 

\begin{figure}[ht]
\begin{center}
\begin{tikzpicture}[scale = 0.85,
    place/.style = {circle,draw = black!50,fill = gray!20,thick,inner sep=3pt},
    auto]

 \node           at (0,1.75) {$\vdots$};
 \node [place] (r1) at (0,1) {};
 \node [place] (n1) at (0,0) [label = right:${\scriptstyle\mathrm{max}-1}$] {};
 
 \node           at (4,1.75) {$\vdots$};
 \node [place] (r2) at (4,1) {};
 \node [place] (n2) at (3,0) [label = right:${\scriptstyle\mathrm{max}}$] {};
 \node [place] (n22) at (5,0) [label = right:${\scriptstyle\mathrm{max}}$] {};
 
 \draw [-,semithick]      (r1) to (n1);
 \draw [-,semithick]      (r2) to (n2);
 \draw [-,semithick]      (r2) to (n22);
 
 \draw [-,semithick]      (n1) to (0.75,-1) to (-0.75,-1) to (n1);
 \draw [-,semithick]      (n2) to (3.75,-1) to (2.25,-1) to (n2);
 \draw [-,semithick]      (n22) to (5.75,-1) to (4.25,-1) to (n22);
 
 \begin{scope}[xshift = 8cm]
 \node [place] (r1) at (0,1) [label = right:${\scriptstyle 2}$] {};
 \node [place] (n1) at (0,0) [label = right:${\scriptstyle 2}$] {};
 
 \node [place] (r2) at (4,1) [label = right:${\scriptstyle 2}$] {};
 \node [place] (n2) at (3,0) [label = right:${\scriptstyle 1}$] {};
 \node [place] (n22) at (5,0) [label = right:${\scriptstyle 1}$] {};
 
 \draw [-,semithick]      (r1) to (n1);
 \draw [-,semithick]      (r2) to (n2);
 \draw [-,semithick]      (r2) to (n22);
 
 \draw [-,semithick]      (n1) to (0.75,-1) to (-0.75,-1) to (n1);
 \draw [-,semithick]      (n2) to (3.75,-1) to (2.25,-1) to (n2);
 \draw [-,semithick]      (n22) to (5.75,-1) to (4.25,-1) to (n22);
 \end{scope}
 
\end{tikzpicture}
 \caption{Forbidden subtrees defining $\beta(1,0)$-trees corresponding to $3$-face-free maps}
 \label{fig:3-ff}
\end{center}
\end{figure}
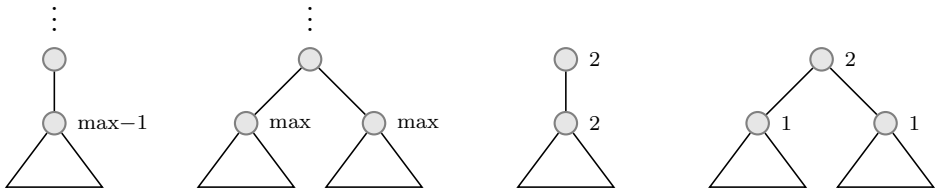

Similarly, one can list forbidden subtrees corresponding to  the $4$-face-free case, which is shown in Figure~\ref{fig:4-ff}. Moreover, based on the cases $k=2,3,4$, it is not difficult to see how to characterize the forbidden subtrees giving $k$-face-free maps: it contains all possible combinations for a node having $m$, $1\leq m\leq k-1$, children with the sum of their labels $k-m-1$ less than the maximum possible sum, and  all $\beta(1,0)$-trees with root labels equal $k$. 

\begin{figure}[ht]
\begin{center}
\begin{tikzpicture}[scale = 0.85,
    place/.style = {circle,draw = black!50,fill = gray!20,thick,inner sep=3pt},
    auto]

 \node           at (3,1.75) {$\vdots$};
 \node [place] (r1) at (3,1) {};
 \node [place] (n1) at (3,0) [label = right:${\scriptstyle\mathrm{max}-2}$] {};
 
 \node           at (6,1.75) {$\vdots$};
 \node [place] (r2) at (6,1) {};
 \node [place] (n2) at (5,0) [label = right:${\scriptstyle\mathrm{max}}$] {};
 \node [place] (n22) at (7,0) [label = right:${\scriptstyle\mathrm{max}-1}$] {};
 
 \node           at (10,1.75) {$\vdots$};
 \node [place] (r3) at (10,1) {};
 \node [place] (n3) at (9,0) [label = right:${\scriptstyle\mathrm{max}-1}$] {};
 \node [place] (n33) at (11,0) [label = right:${\scriptstyle\mathrm{max}}$] {};
 
 \node           at (15,1.75) {$\vdots$};
 \node [place] (r4) at (15,1) {};
 \node [place] (n4) at (13,0) [label = right:${\scriptstyle\mathrm{max}}$] {};
 \node [place] (n44) at (15,0) [label = right:${\scriptstyle\mathrm{max}}$] {};
 \node [place] (n444) at (17,0) [label = right:${\scriptstyle\mathrm{max}}$] {};
 
 \draw [-,semithick]      (r1) to (n1);
 \draw [-,semithick]      (r2) to (n2);
 \draw [-,semithick]      (r2) to (n22);
 \draw [-,semithick]      (r3) to (n3);
 \draw [-,semithick]      (r3) to (n33);
 \draw [-,semithick]      (r4) to (n4);
 \draw [-,semithick]      (r4) to (n44);
 \draw [-,semithick]      (r4) to (n444);
 
 \draw [-,semithick]      (n1) to (3.75,-1) to (2.25,-1) to (n1);
 \draw [-,semithick]      (n2) to (5.75,-1) to (4.25,-1) to (n2);
 \draw [-,semithick]      (n22) to (7.75,-1) to (6.25,-1) to (n22);
 \draw [-,semithick]      (n3) to (9.75,-1) to (8.25,-1) to (n3);
 \draw [-,semithick]      (n33) to (11.75,-1) to (10.25,-1) to (n33);
 \draw [-,semithick]      (n4) to (13.75,-1) to (12.25,-1) to (n4);
 \draw [-,semithick]      (n44) to (15.75,-1) to (14.25,-1) to (n44);
 \draw [-,semithick]      (n444) to (17.75,-1) to (16.25,-1) to (n444);
 
 \begin{scope}[yshift = -3.5cm]
 \node [place] (r1) at (3,1) [label = right:${\scriptstyle 3}$] {};
 \node [place] (n1) at (3,0) [label = right:${\scriptstyle 3}$] {};
 
 \node [place] (r2) at (6,1) [label = right:${\scriptstyle 3}$] {};
 \node [place] (n2) at (5,0) [label = right:${\scriptstyle 1}$] {};
 \node [place] (n22) at (7,0) [label = right:${\scriptstyle 2}$] {};
 
 \node [place] (r3) at (10,1) [label = right:${\scriptstyle 3}$] {};
 \node [place] (n3) at (9,0) [label = right:${\scriptstyle 2}$] {};
 \node [place] (n33) at (11,0) [label = right:${\scriptstyle 1}$] {};
 
 \node [place] (r4) at (15,1) [label = right:${\scriptstyle 3}$] {};
 \node [place] (n4) at (13,0) [label = right:${\scriptstyle 1}$] {};
 \node [place] (n44) at (15,0) [label = right:${\scriptstyle 1}$] {};
 \node [place] (n444) at (17,0) [label = right:${\scriptstyle 1}$] {};
 
 \draw [-,semithick]      (r1) to (n1);
 \draw [-,semithick]      (r2) to (n2);
 \draw [-,semithick]      (r2) to (n22);
 \draw [-,semithick]      (r3) to (n3);
 \draw [-,semithick]      (r3) to (n33);
 \draw [-,semithick]      (r4) to (n4);
 \draw [-,semithick]      (r4) to (n44);
 \draw [-,semithick]      (r4) to (n444);
 
 \draw [-,semithick]      (n1) to (3.75,-1) to (2.25,-1) to (n1);
 \draw [-,semithick]      (n2) to (5.75,-1) to (4.25,-1) to (n2);
 \draw [-,semithick]      (n22) to (7.75,-1) to (6.25,-1) to (n22);
 \draw [-,semithick]      (n3) to (9.75,-1) to (8.25,-1) to (n3);
 \draw [-,semithick]      (n33) to (11.75,-1) to (10.25,-1) to (n33);
 \draw [-,semithick]      (n4) to (13.75,-1) to (12.25,-1) to (n4);
 \draw [-,semithick]      (n44) to (15.75,-1) to (14.25,-1) to (n44);
 \draw [-,semithick]      (n444) to (17.75,-1) to (16.25,-1) to (n444);
 \end{scope}
 
\end{tikzpicture}
 \caption{Forbidden subtrees defining $\beta(1,0)$-trees corresponding to $4$-face-free maps}
 \label{fig:4-ff}
\end{center}
\end{figure}
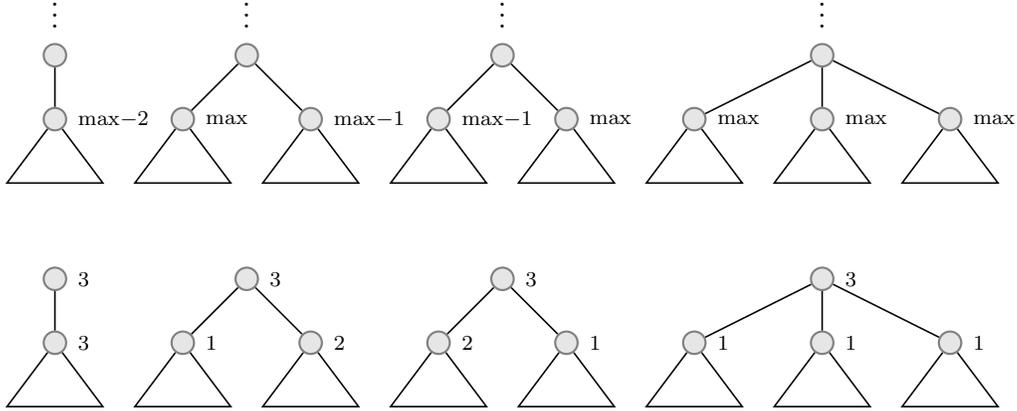

\subsection{Multiple-edge-free maps}

The goal of this subsection is to derive bounds on the number of multiple-edge-free
maps.

Note that $2$-face-free maps can have multiple edges, as illustrated by the map in Figure~\ref{2ff-mult}. However, being $2$-face-free is a necessary condition for a map to be multiple-edge-free, and thus the forbidden subtrees in Figure~\ref{fig:2-ff} are necessary conditions for a $\beta(1,0)$-tree to correspond to a multiple-edge-free map. Thus the generating function for $2$-face-free maps provides an upper bound to the number of multiple-edge-free maps. An asymptotic estimate of the
coefficients can be obtained from Theorem~\ref{enumPrimMaps}, giving
\[
[x^n]P'(x) \sim \frac{529}{1458}\sqrt{\frac{23}{\pi n^3}} \left(\frac{23}{4}\right)^n
\]
These necessary conditions can be extended by one more forbidden structure when an internal node has a single child whose label is 1. Indeed, in this case, when adding in the map the directed edge corresponding to such an internal node with the beginning at the starred vertex and the end at the previous root, we would be creating a multiple edge. Thus we have three forbidden structures giving necessary (but not sufficient) conditions on $\beta(1,0)$-trees to correspond to multiple-edge-free maps.

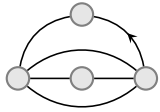
\begin{figure}[ht]
\begin{center}
\begin{tikzpicture}[scale = 0.85, place/.style = {circle,draw = black!50,fill = gray!20,thick,inner sep=3pt}, auto,
  decoration={
    markings,
    mark=at position 0.5 with {\arrow[black]{stealth};}}
    ]
 
 \node [place] (L) at (180:1) {};
 \node [place] (M) at (0,0)   {};
 \node [place] (R) at (0:1)   {};
 \node [place] (T) at (90:1)  {};
 
 \draw [-,semithick, bend left=40]                          (L) to (R);
 \draw [-,semithick]                                        (L) to (M);
 \draw [-,semithick]                                        (M) to (R);
 \draw [-,semithick, postaction={decorate}, bend right=35]  (R) to (T);
 \draw [-,semithick, bend right=35]                         (T) to (L);
 \draw [-,semithick, bend right=40]                         (L) to (R);
 
\end{tikzpicture}
 \caption{A $2$-face free map with multiple edges}
 \label{2ff-mult}
\end{center}
\end{figure}

On the other hand, Lemma~\ref{lemmaSuf} below gives a sufficient condition for a map to be multiple-edge-free.

\begin{lemma}\label{lemmaSuf}
If each internal node in a $\beta(1,0)$-tree has a sibling then the corresponding map is multiple-edge-free.
\end{lemma}

\begin{proof}
Following the recursive way to build the map corresponding to a $\beta(1,0)$-tree where each node has a sibling, we see that edges are only added between vertices that before were only connected by paths going through a cut vertex. Thus, the edge $(u,v)$ cannot be a multiple edge.
\end{proof}

Based on the sufficient condition in Lemma~\ref{lemmaSuf}, one can get a rough lower bound on the number of multiple-edge-free maps. Indeed, assume that all labels, possibly except for the root, are 1. Then we can erase all the labels and use the known generating function for rooted trees where each node has a sibling (see~\cite[A005043]{OEIS}):
\[
B_1(x) = \frac{1+x-\sqrt{1-2x-3x^2}}{2(1+x)}.
\]
Analyzing the singularities of this function and applying Theorem VI.4 in~\cite{FS}
we find that the coefficients are approximately
\[
[x^n]B_1(x) = \frac{1}{8}\sqrt{\frac{3}{\pi n^3}} 3^n + O(3^n n^{-2}).
\]
The difference between the approximation and the actual coefficient of the
generating function is within $0.1\%$ on the $1000$th term.

We can improve the lower bound above by deriving an explicit generating function for the number of $\beta(1,0)$-trees where each internal node has a sibling and all labels, possibly except for the root, are at most 2 (such trees still correspond to multiple-edge-free maps). Let $B_2(x)$ be the generating function for these trees ($x$ is responsible for the number of nodes; we define the coefficient of $x^0$ corresponding to the empty tree to be 0). Then $B_2=B_2(x)$ satisfies the equation:
\[
B_2 = x + B_2(B_2-x) + (B_2-x)(B_2-x) + x( B_2 + (B_2-x))^2.
\]
This can be seen by considering what a tree of this kind can be (see Figure~\ref{fig:lab2}) and using the basic generating functions techniques. Indeed, either a tree is a single node (corresponding to $x$) or its root has degree 2 or more. When the degree of the root is more than 2 (see the middle two trees in Figure~\ref{fig:lab2}), the root's leftmost child, say $u$, has either label 1 or 2. In the first case, the generating function for the number of such trees is given by  $B_2(B_2-x)$ since for the subtree starting at $u$ we can pick any valid $\beta(1,0)$-tree and rewrite its root to be 1 (this gives the factor of $B_2$), and independently, the remaining part of the tree must be a valid $\beta(1,0)$-tree different from a single node (otherwise $u$ would have no sibling; this gives the factor of $B_2-x$). In the second case, the generating function is given by
$(B_2-x)(B_2-x)$ since for the subtree starting at $u$ we now must pick a valid $\beta(1,0)$-tree different from a single node. Finally, if the root's degree is 2 (the rightmost tree in Figure~\ref{fig:lab2}) then its children, independently from each other, can either have label 1 or 2. In the case when the label is 1, the number of choices for the corresponding subtree is given by the generating function $B_2$, while in the case when the label is 2, we have that our choices are given by $B_2-x$ (the single node tree must be excluded, since it will not give the root label 2, while any other legal tree has root label at least 2 which can be rewritten to be 2). Thus, for each of the two subtrees we have independently $B_2+(B_2-x)$ choices, and we add the factor of $x$ to take into account the tree's root.  

We can easily solve the functional equation, which gives
\[
B_2(x) = \frac{1+3x+4x^2 - \sqrt{1-2x-7x^2}}{4+8x}.
\]
Again we can apply Theorem VI.4 from~\cite{FS} to see that the coefficients
are approximately
\[
[x^n]B_2(x) = \frac{1}{8\sqrt{2}+12}\sqrt{\frac{4+\sqrt{2}}{\pi n^3}} \left(\frac{7}{2\sqrt{2}-1}\right)^{n}
+ O\left(\left(\frac{7}{2\sqrt{2}-1}\right)^{n} n^{-2} \right)
\]
The difference between the approximation and the actual coefficient of $B_2$ is
within $0.1\%$ on the $1000$th term.
\begin{figure}[ht]
\begin{center}
\begin{tikzpicture}[scale = 0.85,
    place/.style = {circle,draw = black!50,fill = gray!20,thick,inner sep=3pt},
    auto]

 \node [place] (r1) at (1,1) {};
 
 \node [place] (r2) at (4,1) {};
 \node [place] (n2) at (3,0) [label = right:${\scriptstyle 1}$] {};
 
 \node [place] (r3) at (8,1) {};
 \node [place] (n3) at (7,0) [label = right:${\scriptstyle 2}$] {};
 
 \draw [-,semithick]      (r2) to (n2);
 \draw [-,semithick]      (r3) to (n3);
 
 \draw [-,semithick]      (n2) to (3.75,-1) to (2.25,-1) to (n2);
 \draw [-,semithick]      (n3) to (7.75,-1) to (6.25,-1) to (n3);
 \begin{scope}[xshift = 4cm, yshift = 1cm, rotate = 40]
 \draw [-,semithick]      (r2) to (0.75,-1.5) to (-0.75,-1.5) to (r2);
 \end{scope}
 \begin{scope}[xshift = 8cm, yshift = 1cm, rotate = 40]
 \draw [-,semithick]      (r3) to (0.75,-1.5) to (-0.75,-1.5) to (r3);
 \end{scope}
 
 \node [place] (r4) at (12,1) {};
 \node [place] (n4) at (11,0) [label = right:${\scriptstyle \{1, 2\}}$] {};
 \node [place] (n44) at (13,0)[label = right:${\scriptstyle \{1, 2\}}$] {};
 
 \draw [-,semithick]      (r4) to (n4);
 \draw [-,semithick]      (r4) to (n44);
 
 \draw [-,semithick]      (n4) to (11.75,-1) to (10.25,-1) to (n4);
 \draw [-,semithick]      (n44) to (13.75,-1) to (12.25,-1) to (n44);
 
\end{tikzpicture}
 \caption{The structure of $\beta(1,0)$-trees with labels at most $2$, excluding the root}
 \label{fig:lab2}
\end{center}
\end{figure}
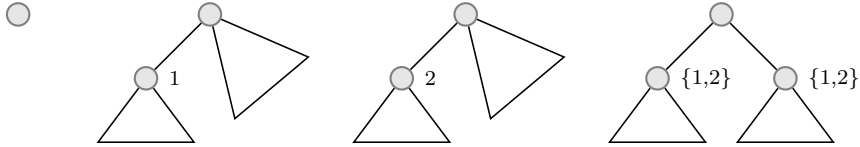
\newline

Even though it becomes quickly cumbersome and difficult to deal with, one can try to handle some larger cases to get a better lower bound for the number of multiple-edge-free maps. For example,  assume that all labels, possibly except for the root are at most 3 and let $B_3(x)$ be the generating function for these trees. Then looking at Figure~\ref{fig:lab3} and applying analysis similar to that applied to Figure~\ref{fig:lab2}, one can conclude that $B_3=B_3(x)$ satisfies the functional equation of degree $4$:
\begin{eqnarray*}
B_3 & = & x + B_3(B_3-x) + (B_3-x)(B_3-x) + (B_3-x-xB_3^2)(B_3-x)\\
    &   & +\ x(3B_3-2x-xB_3^2)^2 \\
    & = & (x + 2x^2 + 4x^3) - (1 + 5x + 12x^2)B_3 + (3 + 9x + x^2 + 4x^3)B_3^2\\
    &   & -\ (x + 6x^2)B_3^3 + x^3B_3^4.
\end{eqnarray*}
This equation can be solved to find a very complicated expression for $B_3$
and used to find the coefficients of the generating function.
The sequence we find starts with $1,0,1,1,5,13,48,160,578,2078$.
Alternatively, let $\phi(B_3)$ be a solution for $x$ from the functional equation. We can find the asymptotics of the coefficients of $B_3(x)$ by estimating the root of
$\phi(z) - z \cdot \phi'(z)$ closest to the origin in the complex plain.
Using the estimate $\tau = 0.28525$ and applying Theorem VIII.8
from~\cite{FS} we find the approximation
\[
[x^n]B_3(x) \sim \gamma \frac{\rho^n}{2\sqrt{\pi n^3}}, \quad \text{where } \rho = \frac{\phi(\tau)}{\tau} \approx 4.24121,\, \gamma = \sqrt{\frac{2\phi(\tau)}{\phi''(\tau)}} \approx 0.12347.
\]
The difference between the approximation and the actual coefficient of $B_3$ is
within $0.1\%$ on the $100$th term.

\begin{figure}[ht]
\begin{center}
\begin{tikzpicture}[scale = 0.85,
    place/.style = {circle,draw = black!50,fill = gray!20,thick,inner sep=3pt},
    auto]

 \node [place] (r1) at (1,1) {};
 
 \node [place] (r2) at (4,1) {};
 \node [place] (n2) at (3,0) [label = right:${\scriptstyle 1}$] {};
 
 \node [place] (r3) at (8,1) {};
 \node [place] (n3) at (7,0) [label = right:${\scriptstyle 2}$] {};
 
 \begin{scope}[xshift = 2.5cm, yshift = -2cm]
 \node [place] (r4) at (0,0) {};
 \node [place] (n4) at (-1,-1) [label = right:${\scriptstyle 3}$] {};

 \draw [-,semithick]       (r4) to (n4);

 \draw [-,semithick]     (n4) to (-0.25,-2) to (-1.75,-2) to (n4);
 \begin{scope}[rotate = 40]
 \draw [-,semithick]      (r4) to (0.75,-1.5) to (-0.75,-1.5) to (r4);
 \end{scope}
 \end{scope}

  \begin{scope}[xshift = 6cm, yshift = -2cm]
 \node [place] (r5) at (0,0) {};
 \node [place] (n5) at (-1,-1) [label = right:${\scriptstyle \{1, 2, 3\}}$] {};
 \node [place] (n55) at (1,-1) [label = right:${\scriptstyle \{1, 2, 3\}}$] {};
 
 \draw [-,semithick]      (r5) to (n5);
 \draw [-,semithick]      (r5) to (n55);
 
 \draw [-,semithick]      (n5) to (-0.25,-2) to (-1.75,-2) to (n5);
 \draw [-,semithick]      (n55) to (1.75,-2) to (0.25,-2) to (n55);
 \end{scope}

 \draw [-,semithick]      (r2) to (n2);
 \draw [-,semithick]      (r3) to (n3);
 
 \draw [-,semithick]      (n2) to (3.75,-1) to (2.25,-1) to (n2);
 \draw [-,semithick]      (n3) to (7.75,-1) to (6.25,-1) to (n3);
 \begin{scope}[xshift = 4cm, yshift = 1cm, rotate = 40]
 \draw [-,semithick]      (r2) to (0.75,-1.5) to (-0.75,-1.5) to (r2);
 \end{scope}
 \begin{scope}[xshift = 8cm, yshift = 1cm, rotate = 40]
 \draw [-,semithick]      (r3) to (0.75,-1.5) to (-0.75,-1.5) to (r3);
 \end{scope}
 
\end{tikzpicture}
 \caption{The structure of $\beta(1,0)$-trees with labels at most $3$, excluding the root}
 \label{fig:lab3}
\end{center}
\end{figure}
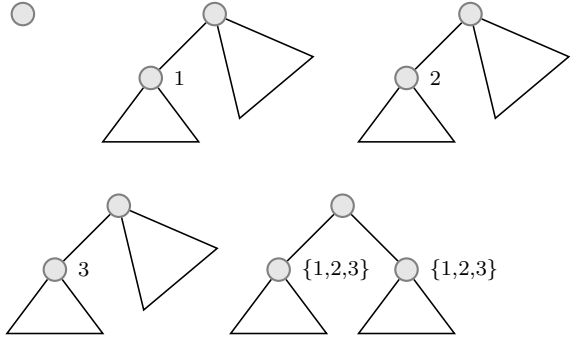

In summary we have found a lower bound for the number of multiple-edge-free maps
with asymptotic growth $(4.24121)^n$
and an upper bound with asymptotic growth $\left(\frac{23}{4}\right)^n = 5.75^n$.

\section{Connections with pattern avoidance in permutations} \label{sec:patts}

In this section we provide a link between our studies and pattern-avoiding permutations (see~\cite{S} for a comprehensive reference book on the area). We need the following definitions.

A permutation $\pi = \pi_1\pi_2\dotsm\pi_n$ is a bijection $\{1,\dotsc,n\} \to \{1,\dotsc,n\}$ that sends
$1$ to $\pi_1$, $2$ to $\pi_2$, $\dotsc$, $n$ to $\pi_n$. Given a string of distinct integers $v = v_1 v_2 \dotsm v_k$
we denote by $\fl(v)$, the \emph{flattening} of $v$, that is, the string where the smallest letter in $v$ has been replaced with $1$, the second smallest letter
replaced with $2$, etc. We say that $\pi$ \emph{contains} another permutation $p = p_1 p_2 \dotsm p_k$ if there are indices
$1 \leq i_1 < i_2 < \dotsb < i_k \leq n$ such that $\fl(\pi_{i_1}\pi_{i_2}\dotsm\pi_{i_k}) = p$. In this context we call $p$ a \emph{classical pattern}.
If $\pi$ does not contain $p$ we say that $\pi$ \emph{avoids} $p$. For example, $32541$ avoids $3142$, while $462531$ contains
one occurrence of the pattern $3142$, namely, the subsequence $4253$.

We will need two generalizations of classical patterns here. The first is vincular patterns, introduced by Babson and Steingrimsson~\cite{BS} (where they were
called \emph{dashed patterns}). In a \emph{vincular} pattern two adjacent letters may or may not be underlined. Underlined letters
mean that the corresponding letters in an occurrence must be adjacent in the permutation. For example, the permutation $253164$
avoids the pattern $2\underline{41}3$, while $365241$ contains one occurrence of the pattern, namely, the subsequence $3524$.

The second generalization is \emph{mesh patterns}, introduced by Br{\"a}nd{\'e}n and Claesson~\cite{BC} to provide explicit expansions
for certain permutation statistics as, possibly infinite, linear combinations of classical permutation patterns. We will only
explain two particular mesh pattern here.
\[
M = \pattern{scale=1}{ 2 }{ 1/2, 2/1 }{1/0,1/1,1/2,2/1} \qquad
M' = \pattern{scale=1}{ 2 }{ 1/2, 2/1 }{1/0,1/1,1/2,2/1,0/2,2/2}
\]
A permutation $\pi$ contains the mesh pattern $M$ if there is an index $i$ such that $\pi_i>\pi_{i+1}$, with the additional
requirement that everything to the right of $\pi_{i+1}$ is either larger than $\pi_i$ or smaller than $\pi_{i+1}$. A permutation
$\pi$ contains the mesh pattern $M'$ if we additionally require that $\pi_{i}$ is the largest element in the permutation. We refer the reader to~\cite{BC} for the
general definition.

We let $\Av(3142,2\underline{41}3)$ denote the set of all permutations avoiding the patterns $3142$ and $2\underline{41}3$ simultaneously.
Claesson et al.~\cite{CKS} proved that permutations of length $n$ in this set are in one-to-one correspondence with $\beta(1,0)$-trees on $n$
edges, and thus with rooted non-separable planar maps on $n+1$ edges. This is the reason the permutations in $\Av(3142,2\underline{41}3)$ are
called the \emph{non-separable permutations}. Claesson et al.~\cite{CKS} actually provided two bijections between the permutations and the trees.
Here we will use the one discussed in the remark on p.~322 in~\cite{CKS}. This bijection is less powerful in terms of the statistics preserved,
but sufficient for our purposes.

The idea behind the bijection is that $\beta(1,0)$-trees and $(3142,2\underline{41}3)$-avoiding permutations can be generated iteratively in the same way, which will induce a bijection sending indecomposable objects to indecomposable ones and decomposable objects to decomposable ones.
This is illustrated by Figures~\ref{fig:indecomp_trees_to_decomp_tree}--\ref{fig:perm_to_indecomp_perm}. We skip justifying that all the steps work, in particular, that in dealing with permutations we do not create a prohibited pattern.

As was mentioned above, a $\beta(1,0)$-tree is indecomposable if the root has only one child, and decomposable if the root has more
than one child. The tree with a single node is neither decomposable nor indecomposable. A permutation $\pi_1\pi_2\cdots\pi_n$ is called \emph{indecomposable} if there is no $i$, $1<i\leq n$, such that $\pi_j<\pi_k$ for all $1\leq j<i\leq k\leq n$ (in other words, an indecomposable permutation does not have a place so that everything to the left of it is smaller than everything to the right of it). This property can be stated with mesh patterns: Indecomposable permutations are the avoiders of the mesh pattern below.
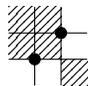
\[
\pattern{scale=1}{ 2 }{ 1/1, 2/2 }{0/1,0/2,1/1,1/2,2/0}
\]
The analogue of the one node tree is the empty permutation.

\noindent
{\bf Generating $\beta(1,0)$-trees on $n$ edges.} A decomposable tree on $n$-edges can be constructed by taking several indecomposable trees on a total of $n$ edges and gluing all their roots into a single node whose label is defined to be the sum of its children. This is shown schematically on three indecomposable trees $T_1$, $T_2$ and $T_3$ in Figure~\ref{fig:indecomp_trees_to_decomp_tree}.
\begin{figure}[ht]
\begin{center}
\begin{tikzpicture}[scale = 0.85,
    place/.style = {circle,draw = black!50,fill = gray!20,thick,inner sep=3pt},
    auto]

 \node [place] (r1) at (0,1) {};
 \node [place] (n1) at (0,0) {};
 
 \node [place] (r2) at (3,1) {};
 \node [place] (n2) at (3,0) {};
 
 \node [place] (r3) at (6,1) {};
 \node [place] (n3) at (6,0) {};
 
 \draw [-,semithick]      (r1) to (n1);
 \draw [-,semithick]      (r2) to (n2);
 \draw [-,semithick]      (r3) to (n3);
 
 \draw [-,semithick]      (n1) to (0.75,-1) to (-0.75,-1) to (n1);
 \node at (0,-0.65) {$T_1$};
 
 \draw [-,semithick]      (n2) to (3.75,-1) to (2.25,-1) to (n2);
 \node at (3,-0.65) {$T_2$};
 
 \draw [-,semithick]      (n3) to (6.75,-1) to (5.25,-1) to (n3);
 \node at (6,-0.65) {$T_3$};
 
 \node at (1.5,0.5) {$\oplus$};
 \node at (4.5,0.5) {$\oplus$};
 
 \draw [->, thick] (7,0.5) to (8,0.5);
 
 \begin{scope}[xshift = 11cm]
 \node [place] (r) at (0,1) {};
 \node [place] (n1) at (-2,0) {};
 \node [place] (n2) at (0,0) {};
 \node [place] (n3) at (2,0) {};
 
 \draw [-,semithick]      (r) to (n1);
 \draw [-,semithick]      (r) to (n2);
 \draw [-,semithick]      (r) to (n3);
 
 \draw [-,semithick]      (n1) to (-1.25,-1) to (-2.75,-1) to (n1);
 \node at (-2,-0.65) {$T_1$};
 
 \draw [-,semithick]      (n2) to (0.75,-1) to (-0.75,-1) to (n2);
 \node at (0,-0.65) {$T_2$};
 
 \draw [-,semithick]      (n3) to (2.75,-1) to (1.25,-1) to (n3);
 \node at (2,-0.65) {$T_3$};
 
 \end{scope}
 
\end{tikzpicture}
 \caption{Constructing a decomposable $\beta(1,0)$-tree from three indecomposable ones}
 \label{fig:indecomp_trees_to_decomp_tree}
\end{center}
\end{figure}
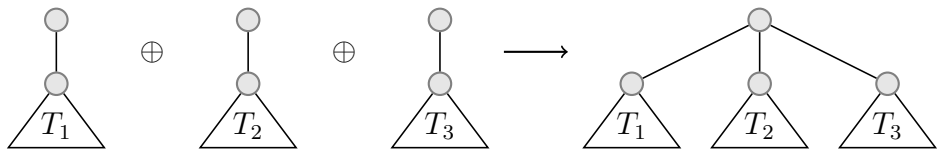

Similarly, an indecomposable tree on $n$ edges with root label $a$ can be obtained by extending a tree on $n-1$ edges with root label $k \geq a$,
by adding one edge as shown in Figure~\ref{fig:tree_to_indecomp_tree}.
\begin{figure}[ht]
\begin{center}
\begin{tikzpicture}[scale = 0.85,
    place/.style = {circle,draw = black!50,fill = gray!20,thick,inner sep=3pt},
    auto]

 \node [place] (r) at (0,1) [label = left:$k$] {};
 
 \draw [-,semithick]      (r) to (0.75,0) to (-0.75,0) to (r);
 
 \draw [->, thick] (2,1) to (3,1);
 
 \begin{scope}[xshift = 5cm]
 \node [place] (r) at (0,1)  [label = left:$a$] {};
 \node [place] (nr) at (0,2) [label = left:$a$] {};
 
 \draw [-,semithick]      (r) to (nr);
 
 \draw [-,semithick]      (r) to (0.75,0) to (-0.75,0) to (r); 
 \end{scope}
 
\end{tikzpicture}
 \caption{Constructing an indecomposable $\beta(1,0)$-tree. Here $1\leq a \leq k$}
 \label{fig:tree_to_indecomp_tree}
\end{center}
\end{figure}
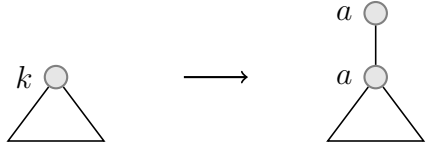

\noindent
{\bf Generating $(3142,2\underline{41}3)$-avoiding permutations of length $n$.} We mimic the steps in the generation of $\beta(1,0)$-trees above. A decomposable permutation of length $n$ can be obtained by taking several indecomposable $(3142,2\underline{41}3)$-avoiding permutations of total length $n$, placing them next to each other and increasing all the letters of the second permutation by the length of the first one, then increasing all the letters of the third permutation by the sum of the lengths of the first and second permutations, etc. This process is shown schematically on three permutations $A_1$, $A_2$ and $A_3$ in Figure~\ref{fig:indecomp_perms_to_decomp_perm}.
\begin{figure}[ht]
\begin{center}
\begin{tikzpicture}[scale = 0.9,
    bx/.style = {rectangle,draw = black!100,thick,inner sep=10pt},
    auto]

 \node [bx] (1) at (0,1) {$A_1$};
 \node [bx] (2) at (2,1) {$A_2$};
 \node [bx] (3) at (4,1) {$A_3$};
 
 \node at (1,1) {$\oplus$};
 \node at (3,1) {$\oplus$};
 
 \draw [->, thick] (5,1) to (6,1);
 
 \begin{scope}[xshift = 7cm]
 \node [bx] (1) at (0,1)      {$A_1$};
 \node [bx] (2) at (1.4,2.3)  {$A_2$};
 \node [bx] (3) at (2.8,3.6)  {$A_3$};
 \end{scope}
 
\end{tikzpicture}
 \caption{Constructing a decomposable $(3142,2\underline{41}3)$-avoiding permutation from indecomposable $(3142,2\underline{41}3)$-permutations}
 \label{fig:indecomp_perms_to_decomp_perm}
\end{center}
\end{figure}
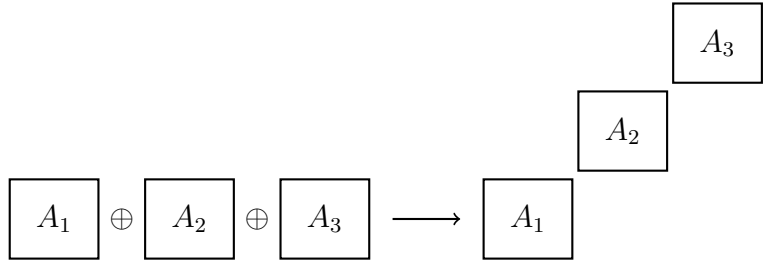

An element in a permutation is a \emph{left-to-right maximum} if there is nothing larger than it to the left of the letter, equivalently, this
element is an occurrence of the mesh pattern $\pattern{scale=0.75}{ 1 }{ 1/1 }{0/1}$.
To create an indecomposable permutation of length $n$, take any $(3142,2\underline{41}3)$-avoiding permutation of length $n-1$ with $k$ left-to-right maxima and first insert the letter $n$ right in front of a left-to-right maximum. In the resulting permutation, let $BnC$ be the right-most indecomposable factor, and let $A$ be the rest; see Figure~\ref{fig:perm_to_indecomp_perm}. While keeping the same relative order inside $A$, $B$ and $C$ we rearrange the factors, as shown in Figure~\ref{fig:perm_to_indecomp_perm}, to build the permutation $\tilde{B}\tilde{A}n\tilde{C}$ with any letter in $\tilde{A}$ larger than any letter in $\tilde{B}$ and $\tilde{C}$. The fact that $\tilde{B}\tilde{A}n\tilde{C}$ is an indecomposable $(3142,2\underline{41}3)$-avoiding permutation follows from Lemma 1 in~\cite{CKS}.
\begin{figure}[ht]
\begin{center}
\begin{tikzpicture}[scale = 0.9,
    bx/.style  = {rectangle,draw = black!100,thick,inner sep=10pt},
    elt/.style = {circle,draw = black!100,fill = black!100,thick,inner sep=2pt},
    auto]

 \node [elt] (lmx1)  at (0,0.5)     {};
 \node [elt]   (lmx2) at (0.75,1)    {};
 \node [elt]   (lmx3) at (1.75,1.25) {};
 \node [elt]   (lmx4) at (2.25,2)    {};
 
 \node [elt]   (n) at    (1.5,2.5) [label = above:{$n$}]  {};
  
 \draw [-,semithick] (0,0) to (lmx1) to (0.75,0.5) to (lmx2) to (1.75,1) to (lmx3) to (2.25,1.25) to (lmx4) to (3,2) to (3,0) to (0,0);
 
 \draw [->,semithick] (n) to (1.5,1.5);
 
 \node [] at (4,1) {$=$};
 
 \begin{scope}[xshift = 5.5cm]
 \node [bx] (A) at (0,0.55)      {$A$};
 \node [bx] (B) at (1.4,1.85)    {$B$};
 \node [bx] (C) at (2.8,1.85)    {$C$};
 
 \node [elt]   (n) at    (2.1,2.75) [label = above:{$n$}]  {};
 
 \draw [-,dashed,semithick] (n) to (2.1,0);
 
 \draw [->, thick] (3.85,1) to (4.35,1);
 \end{scope}
 
 \begin{scope}[xshift = 11cm]
 \node [bx] (B) at (0,0.55)      {$\tilde{B}$};
 \node [bx] (A) at (1.4,1.85)    {$\tilde{A}$};
 \node [bx] (C) at (2.8,0.55)    {$\tilde{C}$};
 
 \node [elt]   (n) at    (2.1,2.75) [label = above:{$n$}]  {};
 
 \draw [-,dashed,semithick] (n) to (2.1,0);
 \end{scope}
 
\end{tikzpicture}
 \caption{Constructing an indecomposable $(3142,2\underline{41}3)$-avoiding permutation from any $(3142,2\underline{41}3)$-permutation}
 \label{fig:perm_to_indecomp_perm}
\end{center}
\end{figure}
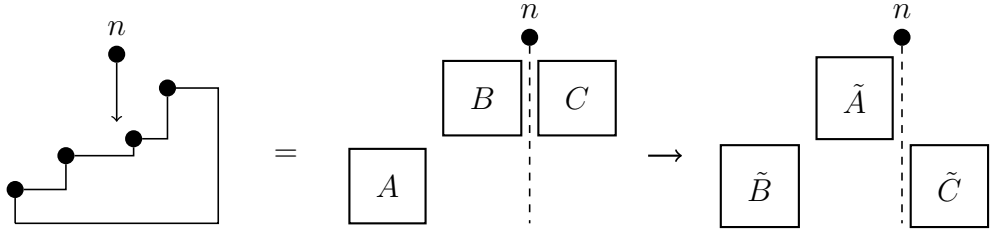

\begin{example}
Consider the permutation $12$. We have two possible locations in which to insert $n = 3$: in front of the $1$ or in front of the $2$. If we choose
the latter then $A = \{1\}$, $B = \emptyset$ and $C = \{2\}$. Thus we get the permutation $231$ out of the procedure. We can further enlarge this
permutation by adding a $4$ in front of the $2$ or the $3$. Choosing the former gives us $A = \emptyset$, $B = \emptyset$, $C = \{231\}$, and the resulting permutation is $4231$. We now only have one choice of extending: Placing $5$ in front of the $4$, so $A = \emptyset$, $B = \{4231\}$, $C = \emptyset$, producing $54231$. On trees the corresponding procedures are shown in Figure~\ref{fig:12example}.
\end{example}

\begin{figure}[ht]
\begin{center}
\begin{tikzpicture}[scale = 0.85, place/.style = {circle,draw = black!50,fill = gray!20,thick,inner sep=3pt}, auto]

 \node [place] (1) at (0,2)  [label = right:$2$]     {};
 \node [place] (2) at (-0.5,1)  [label = left:$1$]       {};
 \node [place] (3) at (0.5,1)  [label = right:$1$]       {};
 
 \draw [-,semithick]      (1) to (2);
 \draw [-,semithick]      (1) to (3);
 
 \draw [->, thick] (1.5,2) to (2.5,2);
 
 \begin{scope}[xshift = 4cm]
 \node [place] (0) at (0,3)  [label = right:$2$]     {};
 \node [place] (1) at (0,2)  [label = right:$2$]     {};
 \node [place] (2) at (-0.5,1)  [label = left:$1$]       {};
 \node [place] (3) at (0.5,1)  [label = right:$1$]       {};
 
 \draw [-,semithick]      (0) to (1);
 \draw [-,semithick]      (1) to (2);
 \draw [-,semithick]      (1) to (3);
 
 \draw [->, thick] (1.5,2) to (2.5,2);
 \end{scope}
 
 \begin{scope}[xshift = 8cm]
 \node [place] (00) at (0,4)  [label = right:$1$]     {};
 \node [place] (0) at (0,3)  [label = right:$1$]     {};
 \node [place] (1) at (0,2)  [label = right:$2$]     {};
 \node [place] (2) at (-0.5,1)  [label = left:$1$]       {};
 \node [place] (3) at (0.5,1)  [label = right:$1$]       {};
 
 \draw [-,semithick]      (00) to (0);
 \draw [-,semithick]      (0) to (1);
 \draw [-,semithick]      (1) to (2);
 \draw [-,semithick]      (1) to (3);

 \draw [->, thick] (1.5,2) to (2.5,2);
 \end{scope}

 \begin{scope}[xshift = 12cm]
 \node [place] (000) at (0,5)  [label = right:$1$]     {};
 \node [place] (00) at (0,4)  [label = right:$1$]     {};
 \node [place] (0) at (0,3)  [label = right:$1$]     {};
 \node [place] (1) at (0,2)  [label = right:$2$]     {};
 \node [place] (2) at (-0.5,1)  [label = left:$1$]       {};
 \node [place] (3) at (0.5,1)  [label = right:$1$]       {};
 
 \draw [-,semithick]      (000) to (00);
 \draw [-,semithick]      (00) to (0);
 \draw [-,semithick]      (0) to (1);
 \draw [-,semithick]      (1) to (2);
 \draw [-,semithick]      (1) to (3);
 \end{scope}
 
\end{tikzpicture}
 \caption{The trees corresponding to the permutations $12$, $231$, $4231$ and $54231$}
 \label{fig:12example}
\end{center}
\end{figure}
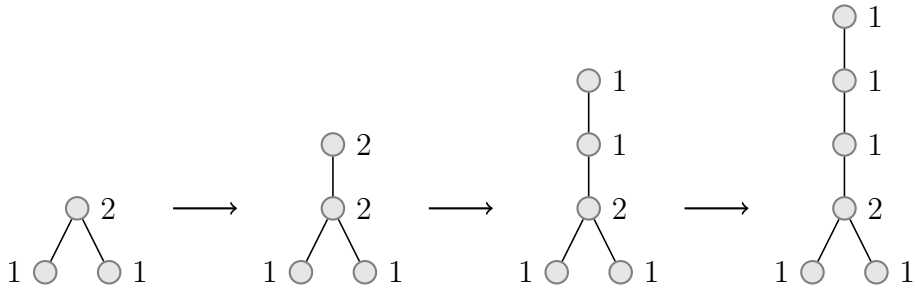

The main result in this section is the following theorem enumerating the set of permutations avoiding one classical, one vincular and one mesh pattern.

\begin{theorem}\label{thm:perms}
Primitive maps, and primitive $\beta(1,0)$-trees, are in one-to-one correspondence with permutations in $\Av(3142,2\underline{41}3)$ that avoid the mesh pattern $M$.
\end{theorem}

This implies that the generating function for the set of permutations $\Av(3142,2\underline{41}3,M)$ is
\[
P(x)=\frac{2(x+1)}{3x}\left(\hypergeom\left(\left[-\frac{2}{3}, -\frac{1}{3}\right],\left[\frac{1}{2}\right],\frac{27x}{4(x+1)}\right)-1\right).
\]

To prove the theorem we first need a lemma.
\begin{lemma} \label{lem:topmax}
Let $T$ be an indecomposable tree. Then the label of the child of the root is the maximum possible if and only if $T$ corresponds to a permutation $\pi$ containing an occurrence of $M'$.
\end{lemma}

\begin{proof}
Suppose the label of the child of the root is a maximum. This means that $n$ was put in right before the last left-to-right-maximum in the permutation. But the last left-to-right maximum is $n-1$ and this will yield the claimed occurrence. Now consider a permutation $\pi$ with an occurrence of $M'$. The right element of the pattern must be the largest element in $\tilde{C}$ (in Figure~\ref{fig:perm_to_indecomp_perm}), which also has to be the rightmost element, and it corresponds to $n-1$ in $C$. This implies that $n$ was put in right before the last left-to-right maximum which means that the child of the root was labeled with the maximum possible value.
\end{proof}

\begin{proof}[Proof of Theorem~\ref{thm:perms}]
Suppose the permutation $\pi$ corresponds to a map that is not primitive. Then the corresponding tree $T$ has at least one node that is a single child and has maximum label. Consider the subtree that contains this node and has the node parent as its root. This subtree is indecomposable and by Lemma~\ref{lem:topmax} corresponds to a permutation containing the pattern $M'$. As the entire tree $T$ is built out of this subtree and others, the occurrence of the pattern $M'$ is either maintained (in the case when the permutation becomes the rightmost component) or becomes an occurrence of $M$ (since the descent will remain a descent).

Now suppose that the permutation $\pi$ contains the pattern $M$ and consider a particular occurrence $xy$ of the pattern $M$. It must lie in a single indecomposable component of $\pi$. If $x$ is the largest letter in its component then the occurrence of $M$ is actually an occurrence of $M'$ inside the component (all the letters of $\pi$, if any, that are larger than $x$ can be ignored). If that is the case then this component corresponds to a tree that is not primitive by Lemma~\ref{lem:topmax}. This would then imply that the entire tree corresponding to $\pi$ is not primitive, completing the proof. However, if $x$ is not the largest letter in the component, we can reverse the construction of $\pi$ (removing largest elements and, possibly, components) until $x$ becomes the largest letter in its component. It then remains to note that reversing the construction maintains descents and does not introduce any elements to the right of a descent, whose size is between the sizes of the descent letters. Thus, $xy$ will correspond to an occurrence of $M'$ in its component at some step of constructing $\pi$, which,  by Lemma~\ref{lem:topmax}, corresponds to a non-primitive tree that will make the entire tree corresponding to $\pi$ non-primitive.
\end{proof}

\begin{lemma} \label{lem:occM}
Let $\pi$ be a permutation in $\Av(3142,2\underline{41}3)$ of length $n-1$ with $m$ occurrences of the mesh pattern $M$. Then inserting $n$ before the last right-to-left maximum in $\pi$ produces a permutation with $m+1$ occurrences of $M$. Inserting $n$ before any other right-to-left maximum produces a permutation with $m$ occurrences of $M$.
\end{lemma}

\begin{proof}
Suppose $\pi$ contains the pattern $M$. Note that the pattern cannot be split between the factors $A$ and $B$ (in Figure~\ref{fig:perm_to_indecomp_perm}) since the underlying classical pattern is $21$,
and it can neither be split between $B$ and $C$ since the left-most element in $C$ is greater than everything in $B$. If the pattern lies entirely in one of the factors $A$, $B$ or $C$ it is easy to see that it will be preserved. Consider, for example, the case when it is in $B$. Then since the relative order of elements in $B$ and $C$ is maintained the pattern will still exist in $\tilde{B}$.
\end{proof}

\begin{theorem}\label{thm:perms-dist}
The number of $2$-faces (excluding the root-face) in a map equals the number of
occurrences of the mesh pattern $M$ in the corresponding
$(3142,2\underline{41}3)$-avoiding permutation, and the number of nodes in the corresponding $\beta(1,0)$-tree that are a single child with maximum label.
\end{theorem}

\begin{proof}
Since an occurrence of the pattern $M$ lies entirely in an indecomposable component of a permutation $\pi$ we see that the number of occurrences of $M$ equals the sum of the numbers of occurrences in each component. The theorem now follows from Lemma~\ref{lem:occM}, since the number of occurrences of $M$ can only be increased by inserting a new largest element in a component in front of the last left-to-right maximum in that component. This corresponds exactly to increasing the number of nodes that are a single child with maximum label.
\end{proof}

A permutation does not start with the largest element if and only if it avoids the mesh pattern $N = \pattern{scale=0.75}{ 1 }{ 1/1 }{0/0,0/1,1/1}$. This gives the following corollary to Theorem \ref{thm:perms-dist}.

\begin{corollary}
Permutations in $\Av(3142,2\underline{41}3,M,N)$ correspond to $2$-face-free maps.
\end{corollary}

\begin{proof}
A permutation starts with the largest element in the case when we have $a=1$ in Figure \ref{fig:tree_to_indecomp_tree}, and this is the prohibited structure to the right in Figure \ref{fig:2-ff}.
\end{proof}

Since the permutations in $\Av(3142,2\underline{41}3,M)$ are in one-to-one correspondence with primitive planar maps and primitive $\beta(1,0)$-trees, it is natural to call these permutations \emph{primitive} among the permutations in $\Av(3142,2\underline{41}3)$.
Thus, one can raise a question of generating all permutations in $\Av(3142,2\underline{41}3)$ from primitive ones, which can be answered as follows.

Given a primitive permutation $\pi$ of length in $\Av(3142,2\underline{41}3)$ and any letter $x$ in $\pi$, a smaller letter $y$
can be inserted adjacent to $x$, creating a descent $xy$, as long as this descent is an occurrence of $M$
that is not involved in an occurrence of the pattern $3142$. Equivalently we can require that $xy$ is an occurrence of at least one
of the mesh patterns below.
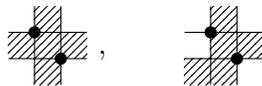
\[
\pattern{scale=1}{ 2 }{ 1/2, 2/1 }{0/1,1/0,1/1,1/2,2/1}, \qquad
\pattern{scale=1}{ 2 }{ 1/2, 2/1 }{0/0,1/0,1/1,1/2,2/1}
\]
Using the procedure above we will create all permutations in $\Av(3142,2\underline{41}3)$. Indeed, given a permutation in $\Av(3142,2\underline{41}3)$, we can destroy all occurrences of the pattern $M$, one by one, by removing the smaller element in each occurrence (it is easy to see that this operation will not create an occurrence of 3142 or $2\underline{41}3$) and ending up with a primitive permutation. Consider for example the permutation $25314$ in $\Av(3142,2\underline{41}3)$. This permutation contains the mesh pattern $M$ in the descent $31$. If we remove the $1$ we end up with a primitive permutation in $\Av(3142,2\underline{41}3)$.

\section{Open questions}

We end with a few open questions.
\begin{enumerate}

\item West-$2$-stack-sortable permutations, first considered by West~\cite{W90}, are in bijection with planar maps, $\beta(1,0)$-trees and the permutations in the set $\Av(3142,2\underline{41}3)$ studied above. It would be interesting to try to understand what describes the primitive permutations among West-$2$-stack-sortable permutations.

\item Using the bijection between $\beta(1,0)$-trees and maps we were not able to determine how multiple edges in a map are manifested in the corresponding $\beta(1,0)$-tree. Is this possible, perhaps with a different bijection? 

\item Similarly, can we understand what structure multiple edges in planar maps correspond to in $(3142,2\underline{41}3)$-avoiding permutations?

\item Finally, the bijection presented in Bonichon et al.~\cite{BBF} between plane bipolar orientations
and Baxter permutations specializes to a bijection between $2413$-avoiding Baxter permutations and
non-separable planar maps. It can thus be seen as a generalization of the bijection in
Dulucq et al~\cite{DGW,DGG}. Can our work be extended to the more general setting of plane bipolar
orientations and Baxter permutations?

\end{enumerate}

\bibliographystyle{alpha}
\bibliography{maps-refs}

\end{document}